\documentclass[10pt]{amsart}

\usepackage{amsmath, amssymb, amsthm, latexsym, amscd}
\usepackage[all]{xy}
\usepackage{graphicx}

 \newlength{\baseunit}               
 \newcount{\numlines}                
\newtheorem*{thm}{Theorem}
\newtheorem{theorem}{Theorem}
\newtheorem{prop}{Proposition}
\newtheorem{lemma}{Lemma}
\newtheorem{corollary}{Corollary}

\theoremstyle{definition}
\newtheorem{define}{Definition}

\newtheorem{example}{Example}

\DeclareMathOperator{\Hom}{Hom}

\DeclareMathOperator{\End}{End}

\DeclareMathOperator{\Id}{Id}

\begin{document}
\pagestyle{plain} 
\title{Positivity and the canonical basis of tensor products of finite-dimensional irreducible representations of quantum $ sl(k)$}
\author{Joshua Sussan}
\date{\today}
\maketitle
\begin{abstract}
In a categorification of tensor products of fundamental representations of quantum $ sl(k) $ via highest weight categories, the indecomposable tilting modules descend to the canonical basis.  Projective functors map tilting modules to tilting modules implying the coefficients of the canonical basis of tensor products of finite dimensional, irreducible representations under the action of the Chevalley generators are positive.  
\end{abstract}

\setcounter{tocdepth}{1}


\section{Introduction}
G. Lusztig gave a geometric construction of the canonical basis in the lower half of the quantized enveloping algebra $ \mathcal{U}^-_q(\mathfrak{g}) $ where $ \mathfrak{g} $ is a finite-dimensional or affine Kac-Moody algebra [Lus].  Due to the geometric nature of this basis, various coefficients associated to this basis are positive and integral.  M. Kashiwara constructed this basis independently using different methods [Kas].
This basis then gives rise to a basis on all irreducible, integrable $ \mathcal{U}_q(\mathfrak{g}) $ modules.  Lusztig then showed how to get a canonical basis for tensor products of such representations.  The goal of this note is to prove that the coefficients of the canonical basis under the action of the Chevalley generators for tensor products of finite dimensional, irreducible representations of $ \mathcal{U}_q(\mathfrak{sl}_k) $ are positive and integral.

In [BFK], $ n $ tensor products of the natural representation of $ \mathfrak{sl}_2 $ were recognized as Grothendieck groups of maximally singular blocks of category $ \mathcal{O}(\mathfrak{gl}_n). $  The action of the Lie algebra was categorified by projective functors acting on these highest weight categories.  In [FKS], this construction was generalized to a categorification of the quantum group on tensor products of arbitrary finite dimensional representations. The categorification in [BFK] was generalized in another direction to a functorial action of $ \mathcal{U}_q(\mathfrak{sl}_k) $ on tensor products of fundamental representations in [Su].  We use this result to prove the following theorem:

\begin{thm}
Let $ \lbrace b_1, \ldots, b_m \rbrace $ be the canonical basis for the tensor product of finite dimensional irreducible $ \mathcal{U}_q(\mathfrak{sl}_k) $ modules $ V(\lambda_1) \otimes \cdots \otimes V(\lambda_k).$ Let $ E_i, F_i, i=1,\ldots,k-1 $ be Chevalley generators for this algebra.  Then $ E_i b_j = \Sigma_n c_{i,j,n} b_n, F_i b_j = \Sigma_n d_{i, j, n} b_n $ where $ c_{i,j,n}, d_{i,j,n} \in \mathbb{N}[q,q^{-1}] $ for all $ i,j,n. $
\end{thm}

Different approaches to categorification of quantums groups have been developed recently [KhLa], [Z]. Using geometric techniques, H. Zheng recently proved this result [Z].  Lusztig proved this theorem for the case of a single irreducible finite dimensional representation in the simply-laced case [Lus].  

The goal is to identify graded lifts of tilting modules in parabolic subcategories of $ \oplus_{\bf d} \mathcal{O}_{\bf d} (\mathfrak{gl}_n) $
as a categorification of the canonical basis of $ \Lambda^{i_1} V_{k-1} \otimes \cdots \otimes \Lambda^{i_r} V_{k-1} $ where $ V_{k-1} $ is the natural representation of $ \mathcal{U}_q(\mathfrak{sl}_k). $  Since projective functors preserve the subcategory of tilting modules, the desired positivity and integrality properties follow.  The theorem above then follows from this special case.

{\it Acknowledgements: } The author is very grateful to Mikhail Khovanov for suggesting this problem and helpful comments on an earlier draft of the paper.

\section{The quantum group and its representations}
\label{definitions}

\begin{define}
The quantum group $ \mathcal{U}_{q}(\mathfrak{sl}_k) $ is the associative algebra over $ \mathbb{Q}(q) $ with generators $ E_i, F_i, K_i, K_{i}^{-1} $ for $ i=1, \ldots, k-1 $ satisfying the following conditions:
\begin{enumerate}
\item $ K_i K_i^{-1} = K_i^{-1} K_i = 1$
\item $ K_i K_j = K_j K_i $
\item $ K_i E_j = q^{c_{i,j}} E_j K_i $
\item $ K_i F_j = q^{-c_{i,j}} F_j K_i $
\item $ E_i F_j - F_j E_i = \delta_{i,j} \frac{K_i - K_i^{-1}}{q-q^{-1}} $
\item $ E_i E_j = E_j E_i $ if $ |i-j|>1 $
\item $ F_i F_j = F_j F_i $ if $ |i-j|>1 $
\item $ E_i^2 E_{i \pm 1} - (q+q^{-1}) E_i E_{i \pm 1} E_i + E_{i \pm 1} E_i^2 = 0 $
\item $ F_i^2 F_{i \pm 1} - (q+q^{-1}) F_i F_{i \pm 1} F_i + F_{i \pm 1} F_i^2 = 0 $
\end{enumerate}
where 
\begin{eqnarray*} 
c_{i,j} = 2 & \textrm{ if } &j=i \\
-1 & \textrm{ if } &j=i \pm1\\
0 & \textrm{if} &|i-j|>1.
\end{eqnarray*}
\end{define}

Let $ V_{k-1} $ be the standard $ k- $ dimensional $ \mathcal{U}_q(\mathfrak{sl}_k)- $ module with basis $ v_0, \ldots, v_{k-1}. $  
The algebra acts on this space as follows:
\begin{eqnarray*}
E_i v_{j} = 0 &\textrm{if}	&j \neq i-1\\
E_i v_{j}= v_i &\textrm{if}	&j=i-1\\
F_i v_j = 0 &\textrm{if} & j \neq i\\ 
F_i v_j = v_{i-1} &\textrm{if} &j=i\\ 
K_i^{\pm 1} v_j = q^{\pm 1} v_i &\textrm{if}	&j=i\\
K_i^{\pm 1} v_{j} = q^{\mp 1} v_{i-1} &\textrm{if} &j=i-1\\ 
K_i^{\pm 1} v_j = v_j &\textrm{if} &j \neq i-1, i.
\end{eqnarray*}

We fix a comultiplication $ \Delta \colon \mathcal{U}_q(\mathfrak{sl}_k) \rightarrow \mathcal{U}_q(\mathfrak{sl}_k) \otimes \mathcal{U}_q(\mathfrak{sl}_k) $ as follows:

\begin{eqnarray*}
\Delta(E_i) = 1 \otimes E_i + E_i \otimes K_i\\
\Delta(F_i) = K_i^{-1} \otimes F_i + F_i \otimes 1\\
\Delta{K_i^{\pm 1}} = K_i^{\pm 1} \otimes K_i^{\pm 1}\\
\end{eqnarray*}

Via $ \Delta, $ a tensor product of $ \mathcal{U}_q(\mathfrak{sl}_k)- $ modules becomes a $ \mathcal{U}_q(\mathfrak{sl}_k)- $ module.

Define the $n- $th fundamental representation $ \Lambda^n V_{k-1} $ to be a quotient of $ V_{k-1}^{\otimes n} $ be the submodule generated by elements of the form 
$$ v_{i_1} \otimes \cdots \otimes v_{i_k} \otimes v_{i_{k+1}} \otimes \cdots \otimes v_{i_n} - q^{-1} v_{i_1} \otimes \cdots \otimes v_{i_{k+1}} \otimes v_{i_k} \otimes \cdots v_{i_n}. $$
Let $ v_{i_1} \wedge \cdots \wedge v_{i_n} $ be the image of the vector $ v_{i_1} \otimes \cdots \otimes v_{i_n} $
for $ i_1 > \cdots > i_n. $

Let $ \pi_n \colon V_{k-1}^{\otimes n} \rightarrow \Lambda^n V_{k-1} $ be the map defined by 
$$ \pi_n(v_{w(i_1)} \otimes \cdots \otimes v_{w(i_n)}) = (-1)^{l(w)-n(n-1)/2} q^{l(w)-n(n-1)/2} v_{i_1} \wedge \cdots \wedge v_{i_n}, $$ 
where $ i_1 > \cdots > i_n, $ $ w \in \mathbb{S}_n, $ and $ l(w) $ is the length of the element $ w. $

Let $ \iota_n \colon \Lambda^n V_{k-1} \rightarrow V_{k-1}^{\otimes n} $ be the map
$$ \iota_n(v_{i_1} \wedge \cdots \wedge v_{i_n}) = \Sigma_{w \in \mathbb{S}_n} (-1)^{l(w)} q^{l(w)} v_{w(i_1)} \otimes \cdots \otimes v_{w(i_n)}, $$
where $ i_1 > \cdots > i_n. $

These give maps: 
$$ \pi_{i_1} \otimes \cdots \pi_{i_r} \colon V_{k-1}^{\otimes (i_1 + \cdots + i_r)} \rightarrow \Lambda^{i_1} V_{k-1} \otimes \cdots \otimes \Lambda^{i_r} V_{k-1} $$ 
$$ \iota_{i_1} \otimes \cdots \otimes \iota_{i_r} \colon \Lambda^{i_1} V_{k-1} \otimes \cdots \otimes \Lambda^{i_r} V_{k-1} \rightarrow V_{k-1}^{\otimes (i_1 + \cdots + i_r)}. $$

Lusztig defined an element $ {\Theta} = 1 + \Sigma a_j \otimes b_j $ in a completion of $ {\mathcal{U}_q(\mathfrak{sl}_k)}^{\otimes 2}, $  
where each $ a_j $ is in the subalgebra generated by the set $ \lbrace E_i | i=1, \ldots,k-1 \rbrace $ and each $ b_j $ is in the subalgebra generated by the set $ \lbrace F_i | i=1, \ldots, k-1 \rbrace. $  See chapter 4 of [Lus] for more details.
This gives rise an involution on tensor products of irreducible, integrable representations.  Let $ V $ and $ W $ be two such representations.  The representations $ V $ and $ W $ have involutions $ \psi_V $ and $ \psi_W $ respectively [Lus].
Then there is an involution $ \psi_{V \otimes W} \colon V \otimes W \rightarrow V \otimes W $ defined by
$ \psi(v \otimes w)=\Theta(\psi_V(v) \otimes \psi_W(w)). $
In particular, this defines
$ \psi $ on $ \Lambda^{i_1} V_{k-1} \otimes \cdots \otimes \Lambda^{i_r} V_{k-1}.  $  On $ \Lambda^{j} V_{k-1}, $ $ \psi(v_{a_1} \wedge \cdots \wedge v_{a_j}) = v_{a_1} \wedge \cdots \wedge v_{a_j}. $  

Lusztig and Kashiwara defined a canonical basis on the quantum algebra and on the tensor products of irreducible, integrable representations.  On the fundamental representation $ \Lambda^n(V_{k-1}), $ the canonical basis is the set of elements of the form $ v_{a_1} \wedge \cdots \wedge v_{a_n}. $
On the representation $ \Lambda^{i_1} V_{k-1} \otimes \cdots \otimes \Lambda^{i_r} V_{k-1}, $ there is a basis of elements of the form
$ b^{i_1}_{j_1} \diamond \cdots \diamond b^{i_r}_{j_r}, $ where $ b^{i_s}_{j_s} $ is a canonical basis element for 
$ \Lambda^{i_s} V_{k-1}. $
This basis is uniquely determined by the following two properties (see chapter 27 of [Lus]):

\begin{equation}
\label{kl1}
\psi(b^{i_1}_{j_1} \diamond \cdots b^{i_r}_{j_r}) = b^{i_1}_{j_1} \diamond \cdots \diamond b^{i_r}_{j_r}
\end{equation}

\begin{equation}
\label{kl2}
b^{i_1}_{j_1} \diamond \cdots \diamond b^{i_r}_{j_r} = b^{i_1}_{j_1} \otimes \cdots \otimes b^{i_r}_{j_r} +
q^{-1} \mathbb{Z}[q^{-1}] \Sigma b^{i_1}_{j_1^*} \otimes \cdots \otimes b^{i_r}_{j_r^*} 
\end{equation} 

where the summation is over all basis elements not equal to $ b^{i_1}_{j_1} \otimes \cdots \otimes b^{i_r}_{j_r}. $

\section{Parabolic-singular category $ \mathcal{O} $}
\subsection{Categorification of $ \Lambda^{i_1} V_{k-1} \otimes \cdots \otimes \Lambda^{i_r} V_{k-1} $}
\label{3.1}

Fix a triangular decomposition $ \mathfrak{gl}_n = \mathfrak{n}^- \oplus \mathfrak{h} \oplus \mathfrak{n}^+ $ where $ \mathfrak{n}^- $ are the lower triangular matrices, $ \mathfrak{n}^+ $ are the upper triangular matrices and $ \mathfrak{h} $ are the diagonal matrices.
Let $ e_1, \ldots, e_n $ be a basis for $ \mathfrak{h} $ with a dual basis $ e_1^*, \ldots, e_n^* $ of $ \mathfrak{h}^*. $

\begin{define}
Let $ \mathcal{O}(\mathfrak{gl}_n) $ be the full subcategory of $ \mathfrak{gl}_n $ modules which satisfy the following properties:
\begin{enumerate}
\item Finitely generated as $ \mathcal{U}(\mathfrak{gl}_n)-  $ modules.
\item Diagonalizable under the action of the Cartan subalgebra $ \mathfrak{h}. $
\item Locally finite under the action of the Borel subalgebra $ {\mathfrak{b}} = {\mathfrak{h}} + {\mathfrak{n}}^{+}. $
\end{enumerate}
\end{define}

This category decomposes into a direct sum of subcategories corresponding to the generalized central characters.

\begin{define}
\label{defblocks}
\begin{enumerate}
\item Let 
$ \mathcal{O}_{(d_{k-1}, d_{k-2}, \ldots, d_{0})}(\mathfrak{gl}_n) = \mathcal{O}_{\bf d}(\mathfrak{gl}_n) $
be the block of $ \mathcal{O}(\mathfrak{gl}_n) $ for the central character corresponding to the weight 
$ \omega_{\bf d} = \sum_{i=0}^{k-1} \sum_{j=1}^{d_i} i e_{d_{k-1} + \cdots + d_{i+1} + j}^{*} - \rho, $
where $ e_{d_{k-1} + \cdots + d_{k} + j}^{*} = e_j^* $ and 
$$ \rho = \frac{n-1}{2}e_1 + \frac{n-3}{2}e_2 + \cdots + \frac{1-n}{2}e_n $$
is half the sum of the positive roots.
\item Let $ \mathcal{O}_0 (\mathfrak{gl}_n) $ be the trivial block.
\item Let $ M(a_1, \ldots, a_n) $ be the Verma module with highest weight $ a_1 e_1+ \cdots + a_n e_n - \rho. $
\item Let $ L(a_1, \ldots, a_n) $ be the simple module with highest weight $ a_1 e_1+ \cdots + a_n e_n - \rho. $
\item Let $ P(a_1, \ldots, a_n) $ be the indecomposable projective cover of $ L(a_1, \ldots, a_n). $
\end{enumerate}
\end{define}

There are $ d_i $ terms in the weight from the definition above with coefficient $ i. $  Note that $ \sum_{j=0}^{k-1} d_j = n. $
For a triangulated category $ \mathcal{C}, $ denote by $ [\mathcal{C}] $ the Grothendieck group of $ \mathcal{C}. $  For a triangulated functor $ \mathcal{J} \colon \mathcal{C} \rightarrow \mathcal{D} $ let $ [\mathcal{J}] $ denote the image of the functor on the Grothendieck group.

\begin{prop}
\label{ungraded}
Assume that the following direct sum is over all $ \bf{d} $ such that the entries are non-negative integers and the sum of the entries is n.  Then
$ \mathbb{Q} \otimes_{\mathbb{Z}} [\oplus_{\bf{d}} \mathcal{O}_{{\bf d}}(\mathfrak{gl}_n)] \cong V_{k-1}^{\otimes n} $
where $ V_{k-1} $ is just a $ k- $ dimensional vector space over $ \mathbb{Q}. $
\end{prop}

\begin{proof}
The image of the Verma module $ [M(a_1, \ldots, a_n)] $ gets mapped to $ v_{a_1} \otimes \cdots \otimes v_{a_n}. $
\end{proof}

This proposition is the first step towards categorification of $ \mathfrak{sl}_k- $ modules.
Next we would like to categorify the action of the Lie algebra.  The desired functors come directly from [BFK].  It is essentially the projective functor of
tensoring with the $ n-$ dimensional representation $ V_{n-1}. $  One only has to be careful about projecting onto the various blocks.  Define $ \text{proj}_{\bf d} $ the functor of projecting onto the block $ \mathcal{O}_{\bf d}(\mathfrak{gl}_n). $

\begin{define}
\begin{enumerate}
	\item Let 
		$ \mathcal{E}_{i} \colon \mathcal{O}_{(d_{k-1},d_{k-2}, \ldots, d_{0})}(\mathfrak{gl}_n) \rightarrow
		\mathcal{O}_{(d_{k-1},\ldots, d_i +1, d_{i-1}-1, \ldots, d_{0})}(\mathfrak{gl}_n) $ be the functor defined by
		$$ \mathcal{E}_{i} M = \text{proj}_{(d_{k-1}, \ldots, d_i+1, d_{i-1}-1, \ldots, d_{0})}
		(V_{n-1} \otimes M). $$
	\item Let 
		$ \mathcal{F}_{i} \colon \mathcal{O}_{(d_{k-1}, d_{k-2}, \ldots, d_{0})}(\mathfrak{gl}_n) \rightarrow
		\mathcal{O}_{(d_{k-1}, \ldots, d_{i}-1, d_{i-1}+1, \ldots, d_{0})}(\mathfrak{gl}_n) $ by
		$$ \mathcal{F}_{i} M = \text{proj}_{(d_{k-1}, \ldots, d_{i}-1, d_{i-1}+1, \ldots, d_{0})}
	 ({V_{n-1}^{*}} \otimes M). $$
	\item Let 
		$ \mathcal{H}_{i} \colon \mathcal{O}_{(d_{k-1}, d_{k-2}, \ldots, d_{0})}(\mathfrak{gl}_n) \rightarrow
		\mathcal{O}_{(d_{k-1}, d_{k-2}, \ldots, d_{0})}(\mathfrak{gl}_n) $ be 
		$ \Id^{\oplus (d_i - d_{i-1})}. $
\end{enumerate}
\end{define}

Let $ I $ and $ J $ be compositions of $ n. $  
If $ I = i_1 + \cdots+ i_r = n, $ associate to $ I $ the Young subgroup of $ \mathbb{S}_n, $ $ \mathbb{S}_{i_1} \times \cdots \times \mathbb{S}_{i_r}. $
Let $ \mu_I $ and $ \mu_J $ be integral dominant weights stabilized by the subgroups associated to $ I $ and $ J.$
Suppose $ J \subset I, $ (there is a containment of the associated subgroups.)
Then we can define the translation functor $ \theta_{\mu_I}^{\mu_J} = \theta_{I}^{J} $ from $ \mathcal{O}_{\mu_I}(\mathfrak{gl}_n) $ to $ \mathcal{O}_{\mu_J}(\mathfrak{gl}_n). $  It is the projective functor given be tensoring with the finite dimensional, irreducible module with highest weight $ \mu_J - \mu_I $ and then projecting onto the block $ \mathcal{O}_{\mu_J}(\mathfrak{gl}_n). $  There is also an adjoint functor $ \theta_J^I $ of tensoring with the dual module and then projecting onto the appropriate block.
Let $ \theta_0^{i} $ be translation from the trivial block onto the $ ith $ wall and let 
$ \theta_{i}^0 $ be translation from the wall back into the trivial block.
Finally, let $ \theta_{i} = \theta_{i}^0 \theta_0^{i}. $

\begin{define}
\begin{enumerate}
\item The subalgebra $ \mathfrak{p}_{(r_1, \ldots, r_t)} $ is the parabolic subalgebra whose reductive subalgebra is
$ \mathfrak{gl}_{r_1} \oplus \cdots \oplus \mathfrak{gl}_{r_t}, $  where $ r_1 + \cdots + r_t = n. $
\item Denote by $ \mathcal{O}_{\bf d}^{\mathfrak{p}}(\mathfrak{gl}_n) $ the full subcategory of $ \mathcal{O}_{\bf d}(\mathfrak{gl}_n) $ of modules locally finite with respect to the subalgebra
$ \mathfrak{p}. $ Let
$ \mathcal{O}_{\bf d}^{(r_1, \ldots, r_t)}(\mathfrak{gl}_n) $ be the category $ \mathcal{O}_{\bf d}^{\mathfrak{p}_{(r_1, \ldots, r_t)}}(\mathfrak{gl}_n). $
\item Let $ Z^{\mathfrak{p}} \colon \mathcal{O}_{\bf d}(\mathfrak{gl}_n) \rightarrow
\mathcal{O}_{\bf d}^{\mathfrak{p}}(\mathfrak{gl}_n) $ be the dual Zuckerman functor of taking the maximal locally finite quotient with respect to $ \mathcal{U}(\mathfrak{p}). $  The corresponding derived functor on the bounded derived category is $ LZ^{\mathfrak{p}}. $
\item Let $ \epsilon_{\mathfrak{p}} \colon \mathcal{O}_{\bf d}^{\mathfrak{p}}(\mathfrak{gl}_n) \rightarrow \mathcal{O}_{\bf d}(\mathfrak{gl}_n) $ be the exact inclusion functor.
\end{enumerate}
\end{define}

We now recall the definition of the generalized Verma modules which are objects in these locally finite categories.  
\begin{define}
\begin{enumerate}
\item Let $ S $ denote the subset of simple roots defining the parabolic subalgebra $ \mathfrak{p}. $
\item Let $ P_{\mathfrak{p}}^{+} = \lbrace \lambda \in \mathfrak{h}^{*} | \langle \lambda, \alpha \rangle \in \mathbb{Z}_{\geq 0}, \forall \alpha \in S \rbrace. $
\end{enumerate}
\end{define}

Given such a $ \lambda \in P_{\mathfrak{p}}^{+}, $ we may define the generalized Verma module
$ M^{\mathfrak{p}}(\lambda) = \mathcal{U}(\mathfrak{g}) \otimes_{\mathcal{U}(\mathfrak{p})} E(\lambda), $ where $ E(\lambda) $
is the simple $ \mathfrak{p}- $ module with highest weight $ \lambda. $

If $ \lambda = a_1 e_{1}^{*} + \cdots + a_n e_{n}^{*} - \rho, $  then $ \lambda \in P_{\mathfrak{p}}^+ $ if $ a_i > a_{i+1} $ whenever $ \alpha_i \in S. $  In that case, set $ M^{\mathfrak{p}}(\lambda)=M^{\mathfrak{p}}(a_1, \ldots, a_n). $

\begin{prop}
There is an isomorphism of vector spaces
$ \mathbb{C} \otimes_{\mathbb{Z}} [\oplus_{\bf d} \mathcal{O}_{\bf d}^{(r_1, \ldots, r_t)}] \cong \Lambda^{r_1} V_{k-1} \otimes \cdots \otimes \Lambda^{r_t} V_{k-1} $
where $ \Lambda^j V_{k-1} $ is an exterior power of $ V_{k-1}. $
\end{prop}

\begin{proof}
Let $ \mathfrak{p} $ be the subalgebra given above.
The isomorphism sends $ [M^{\mathfrak{p}}(a_1, \ldots, a_n)] $ to
$$ (v_{a_1} \wedge \cdots \wedge v_{r_1}) \otimes \cdots \otimes
(v_{a_{r_1 + \cdots + r_{t-1} + 1}} \wedge \cdots \wedge v_{a_{r_1+ \cdots + r_t}}). $$
This is clearly a bijection.
\end{proof}

If any of the $ r_i $ above is larger than $ k, $ then the category contains no non-trivial objects.

\subsection{Tilting modules}
We now introduce a collection of modules which will descend to the canonical basis in the Grothendieck group.  The tilting objects in category $ \mathcal{O} $ were classified by Collingwood and Irving [CI].
\begin{theorem}
\label{tilting}
In $ \mathcal{O}_0(\mathfrak{gl}_n), $ for each $ w \in \mathbb{S}_n, $ there exists a unique (up to isomorphism), indecomposable module $ T(w) $ such that $ T(w) $ is self-dual and $ T(w) $ has a Verma flag with $ M(w) $ occuring as a submodule.
\end{theorem}

\begin{proof}
See [CI].
\end{proof}

Collingwood and Irving generalize this result to a parabolic subcategory of the trivial block.
From now on let $ \mathfrak{p} $ be a parabolic subalgebra of $ \mathfrak{gl}_n $ containing the reductive subalgebra $ \mathfrak{gl}_{i_1} \oplus \cdots \oplus \mathfrak{gl}_{i_r}. $  Consider the corresponding Young subgroup $ \mathbb{S}_{\bf i} = \mathbb{S}_{i_1} \times \cdots \times \mathbb{S}_{i_r}. $  Let $ (\mathbb{S_{\bf i}} \backslash \mathbb{S}_n)_{shortest} $ be the set of shortest coset representatives.  Let $ w_0^{\bf i} $ be the longest element of this set.

\begin{theorem}
\label{tiltingpar}
For each $ w \in (\mathbb{S_{\bf i}} \backslash \mathbb{S}_n)_{shortest}, $ there exists a unique up to isomorphism, self-dual, indecomposable object $ T^{\mathfrak{p}}(w) $ of $ \mathcal{O}_0^{\mathfrak{p}}(\mathfrak{gl}_n) $  which has a filtration by generalized Verma modules such that $ M^{\mathfrak{p}}(w) $ occurs as a submodule.
\end{theorem}

\begin{proof}
See [CI].
\end{proof}

These tilting objects are constructed as direct summands of translation functors applied to the simple Verma module in the case of theorem ~\ref{tilting} and direct summands of translation functors applied to the simple generalized Verma module in the case of theorem ~\ref{tiltingpar}.  Let $ w_0 $ be the longest element of $ \mathbb{S}_n. $

\begin{corollary}
\label{zucktilt}
There is an isomorphism $ LZ^{\mathfrak{p}} T(w w_0) \cong T^{\mathfrak{p}}(w w_0)[l(w_0^{\bf i})], $
when $ w \in (\mathbb{S_{\bf i}} \backslash \mathbb{S}_n)_{shortest}. $ 
If $ w \notin (\mathbb{S}_{\bf i} \backslash \mathbb{S}_n)_{shortest}, $ then $ LZ^{\mathfrak{p}} T(w w_0) = 0. $ 
\end{corollary}

\begin{proof}
Let $ w \in (\mathbb{S}_{\bf i} \backslash \mathbb{S}_n)_{shortest}. $ By [ES], $ LZ^{\mathfrak{p}} M(w_0) \cong M^{\mathfrak{p}}(w_0^{\bf i})[l(w_0^{\bf i})]. $
Since $ T(w w_0) $ is a direct summand of a translation functor applied to $ M(w_0), $ and translation functors naturally commute with the Zuckerman functor, we get the desired isomorphism.

For the case $ w \notin (\mathbb{S}_{\bf i} \backslash \mathbb{S}_n)_{shortest}, $
write $ w = \sigma_{j_1} \cdots \sigma_{j_a} \sigma_{k_1} \cdots \sigma_{k_b} $ where 
$ \sigma_{j_1} \cdots \sigma_{j_a} \in \mathbb{S}_{\bf i} $ and 
$ \sigma_{k_1} \cdots \sigma_{k_b} \in (\mathbb{S}_{\bf i} \backslash \mathbb{S}_n)_{shortest}. $
Since $ w \notin (\mathbb{S}_{\bf i} \backslash \mathbb{S}_n)_{shortest}, $ 
$ \sigma_{j_1} \cdots \sigma_{j_a} \neq e. $
In the notation of [CI], $ T(w w_0) $ is a direct summand of $ \theta(w) M(w_0) $ where
$ \theta(w) = \theta_{{k_b}} \cdots \theta_{{j_1}}. $

Now, $ LZ^{\mathfrak{p}} T(w w_0) \subset LZ^{\mathfrak{p}} \theta(w) M(w_0) = \theta(w) M^{\mathfrak{p}}(w_0^{\bf i})[l(w_0^{\bf i})] = 0  $ since 
$ \theta_{j_1} M^{\mathfrak{p}}(w_0^{\bf i}) = 0 $ by [CI].
\end{proof}

Tilting modules in category $ \mathcal{O}_{\lambda}(\mathfrak{gl}_n) $ and $ \mathcal{O}_{\lambda}^{\mathfrak{p}} (\mathfrak{gl}_n), $ where $ \lambda $ may be singular, are constructed from the tilting modules in the trivial block via translation.
Let $ \theta_0^{\lambda} $ and $ \theta_{\lambda}^0 $ be translations onto and off the wall respectively.

\begin{prop}
The object $ \theta_0^{\lambda} T(w) $ has a unique (up to isomorphism) direct summand with submodule isomorphic to $ M(w.\lambda). $
\end{prop}

\begin{proof}
Assume $ \theta_0^{\lambda}T(w) \cong T_1(w.\lambda) \oplus T_2(w.\lambda) \oplus T_G $ 
where $ T_1(w.\lambda) $ and $ T_2(w.\lambda) $ are indecomposable tilting objects which have submodules isomorphic to $ M(w.\lambda). $
Then 
$ \theta_{\lambda}^0 \theta_0^{\lambda}T(w) \cong \theta_{\lambda}^0 T_1(w.\lambda) \oplus \theta_{\lambda}^{0} T_2(w.\lambda) \oplus \theta_{\lambda}^0 T_G. $ 
The proof of theorem 4.1 of [Maz] shows that translating an indecomposable tilting module off the wall gives an indecomposable tilting module.  Thus $ \theta_{\lambda}^0 M(w.\lambda) \subset T(w_1) $ and $ \theta_{\lambda}^0 M(w.\lambda) \subset T(w_2) $ for some $ w_1 $ and $ w_2. $
Since $ \theta_{\lambda}^0 M(w.\lambda) $ has a Verma flag, there exists a $ w_3 $ such that $ M(w_3) $ is contained in both $ T(w_1) $ and $ T(w_2). $  Thus $ T(w_1) \cong T(w_2) \cong T(w_3). $
Thus $ \theta_{\lambda}^0 T_1(w.\lambda) \cong \theta_{\lambda}^0 T_2(w.\lambda) \cong T(w_3). $
Therefore $ \theta_0^{\lambda} \theta_{\lambda}^0 T_1(w.\lambda) \cong \theta_0^{\lambda} \theta_{\lambda}^0 T_2(w.\lambda). $
Translating off the wall and then back on to it gives a direct sum of copies of the identity functor.  Therefore
$ T_1(w.\lambda) \cong T_2(w.\lambda). $
\end{proof}

\begin{define}
\begin{enumerate}
\item Suppose $ w \in (\mathbb{S}_n / \mathbb{S}_{\lambda})_{shortest}. $ Let $ T(w.\lambda) $ be an indecomposable summand of $ \theta_0^{\lambda}T(w) $ which has $ M(w.\lambda) $ as a submodule.  
\item Suppose the stabilizer of $ \lambda $ is $ \mathbb{S}_{\bf d}. $  Let $ w \in (\mathbb{S}_{\bf i} \backslash \mathbb{S}_n)_{shortest} $ and $ w \in (\mathbb{S}_n /  \mathbb{S}_{\bf d})_{shortest}. $
Let $ T^{\mathfrak{p}}(w.\lambda) $ be an indecomposable summand of $ \theta_0^{\lambda}T^{\mathfrak{p}}(w) $ which has $ M^{\mathfrak{p}}(w.\lambda) $ as a submodule.
\end{enumerate}
\end{define}

\begin{lemma}
The object $ LZ^{\mathfrak{p}}(T(w.\lambda)) $ is a shifted indecomposable tilting object.
\end{lemma}

\begin{proof}
By construction, $ LZ^{\mathfrak{p}}(T(w.\lambda)) $ is direct sum of shifted tilting objects.
Assume $ LZ^{\mathfrak{p}} T(w.\lambda) \cong T^{\mathfrak{p}}(w_1.\lambda)[a] \oplus T^{\mathfrak{p}}(w_2.\lambda)[a]. $
Then $ \theta_{\lambda}^0 LZ^{\mathfrak{p}}(w.\lambda) \cong LZ^{\mathfrak{p}} \theta_{\lambda}^0 T(w.\lambda) \cong LZ^{\mathfrak{p}} T(w') $ for some $ w'. $  This object is indecomposable.
On the other hand, it is isomorphic to
$ \theta_{\lambda}^0 T^{\mathfrak{p}}(w_1.\lambda)[a] \oplus \theta_{\lambda}^0 T^{\mathfrak{p}}(w_2.\lambda)[a]. $
Thus, one of these objects is zero.  Since $ \theta^{\lambda}_0 \theta^0_{\lambda} $ is a direct sum of identity functors, either $ T^{\mathfrak{p}}(w_1.\lambda)[a] $ or $ T^{\mathfrak{p}}(w_2.\lambda)[a] $ is zero.
\end{proof}

\begin{prop}
\begin{enumerate}
\item All indecomposable tilting objects of $ \mathcal{O}_{\lambda}(\mathfrak{gl}_n) $ are of the form $ T(w.\lambda). $
\item All indecomposable tilting objects of $ \mathcal{O}_{\lambda}^{\mathfrak{p}}(\mathfrak{gl}_n) $ are of the form $ T^{\mathfrak{p}}(w.\lambda). $
\end{enumerate}
\end{prop}

\begin{proof}
We only prove the first statement.  The proof of the second statement is similar.
Let $ T $ be an indecomposable tilting object in $ \mathcal{O}_{\lambda}(\mathfrak{gl}_n). $
Then again by [Maz], $ \theta_{\lambda}^0 T \cong T(w) $ for some $ w. $
Then $ \theta_0^{\lambda} \theta_{\lambda}^0 T \cong \oplus T \cong \theta_0^{\lambda} T(w). $ 
Since $ \theta_0^{\lambda} T(w) \cong T(w.\lambda) \oplus \oplus_i T_i $ where $ T_i $ are other indecomposable tilting modules, $ T_i \cong T \cong T(w.\lambda) $ for all $ i. $
\end{proof}

\subsection{Graded Category $ \mathcal{O} $}
In section ~\ref{3.1}, various aspects of the representation theory of $ \mathfrak{sl}_k $ were categorified.   The categorification of quantum groups is accomplished through graded representation theory.  We will treat category $ \mathcal{O} $ as a category of graded modules.  Then a shift in this grading descends to multiplication by $ q $ in the Grothendieck group: $ [M \langle 1 \rangle]=q[M]. $ The idea of graded category $ \mathcal{O} $ originates from [Soe1], [Soe3].  In [Str1], it was shown how to construct graded lifts of translation functors.  

\begin{define}
Let $ P_{\bf d} = \oplus_{x \in \mathbb{S}_n/\mathbb{S}_{\bf d}} P(x. \omega_{\bf d}) $ where $ {\bf d} = (d_{k-1}, \ldots, d_0), $ $ x $ is a minimal coset representative, and
$ \omega_{\bf d} $ which was defined in definition ~\ref{defblocks}.
\end{define}

Then $  P_{\bf d} $ is a minimal projective generator of $ \mathcal{O}_{\bf d}(\mathfrak{gl}_n). $
There is an equivalence of categories $ \mathcal{O}_{\bf d}(\mathfrak{gl}_n) \cong \text{mod}-A_{\bf d} $ where
$ A_{\bf d} = \End_{\mathfrak{g}}(P_{\bf d}) $ and $ \text{mod}-A_{\bf d} $ is the category of finitely generated right $ A_{\bf d}- $ modules.  We will interpret $ A_{\bf d} $ as a graded algebra.

The following lemma may be found in [Bass].

\begin{lemma}
\label{lemma28}
Let $ R $ and $ S $ be any rings.  There is an equivalence of categories:
$$ \lbrace \text{right exact functors compatible with direct sums}: (\text{mod-}R \rightarrow \text{mod-}S) \rbrace \rightarrow R\text{-mod-}S. $$
Under this equivalence a functor $ F $ gets mapped to $ F(R). $ In the other direction,
a bimodule $ X $ gets mapped to $ \bullet \otimes_{R} X. $
\end{lemma}

Let $ P(x.\omega_{\bf d}) $ be an indecomposable projective object in $ \mathcal{O}_{\bf d}(\mathfrak{gl}_n). $ Then $ Z^{\mathfrak{p}} P(x.\omega_{\bf d}) $ is either 0 or an indecomposable projective object in
$ \mathcal{O}_{\bf d}^{\mathfrak{p}} (\mathfrak{gl}_n), $ (see [Ro].) 

Denote by $ P_{\bf d}^{\mathfrak{p}} $ a minimal projective generator of $ \mathcal{O}_{\bf d}^{\mathfrak{p}} (\mathfrak{gl}_n) $. 
Let $ A_{\bf d}^{\mathfrak{p}} = \End(P_{\bf d}^{\mathfrak{p}}) $ be its endomorphism algebra.
Then there is an equivalence of categories
$ \mathcal{O}_{\bf d}^{\mathfrak{p}} (\mathfrak{gl}_n)  \cong \text{mod}-A_{\bf d}^{\mathfrak{p}} $.
Let $ S = S(\mathfrak{h}) $ be the symmetric algebra associated to the Cartan subalgebra.  
Let $ W = \mathbb{S}_n $ be the Weyl group.
Let $ C = S/S_{+}^{W} $ be the associated coinvariant algebra.
Let $ C^{\bf d} $ be the subalgebra invariant under $ W_{\bf d}. $  Let $ w_0^{\bf d} $ be the longest element in the set of shortest coset representatives of $ W/W_{\bf d}. $  There is the is a well known isomorphism due to Soergel [Soe1] between the endomorphism algebra of the indecomposable projective-injective module and this subalgebra of invariants: $ \End(P(w_{0}^{\bf d}.\omega_{\bf d})) \cong C^{\bf d}. $

\begin{define}
Let $ \mathbb{V}_{\bf d} \colon \mathcal{O}_{\bf d}(\mathfrak{gl}_n) \rightarrow \text{mod}-C^{\bf d} $ be the Soergel functor defined by
$ M \mapsto \Hom_{\mathfrak{g}}(P(w_{0}^{\bf d}.\omega_{\bf d}), M). $
\end{define}

Soergel showed that the functor $ \mathbb{V}_{\lambda} $ is fully faithful on projective objects [Soe1].  Suppose $ J \subset I $ as defined earlier.
Translation functors and Soergel functors are related by the following lemma originally proved by Soergel.

\begin{lemma}
\label{lemma29}
Let $ \text{Res}_{J}^{I} \colon \text{mod}-C^J \rightarrow \text{mod}-C^I $ denote the restriction functor.  Then
\begin{enumerate}
\item $ {\mathbb{V}}_J \theta_I^J \cong C^J \otimes_{C^I} {\mathbb{V}}_I. $
\item $ {\mathbb{V}}_I \theta_J^I \cong Res_J^I {\mathbb{V}}_J. $
\end{enumerate}
\end{lemma}

\begin{proof}
This is proposition 3.3 of [FKS].
\end{proof}

This lemma together with the fact that $ \mathbb{V}_{\bf d} $ is a faithful functor on projective objects allows us to consider the endomorphism ring of a minimal projective generator of $ \mathcal{O} $ as a graded ring [Str1].  A projective object is a direct summand of a sequence of projective functors applied to a dominant Verma module.  Thus by the previous lemma, a projective object $ P, $ $ \mathbb{V}_{\bf d} P $ becomes a graded $ C^{\bf d}- $ module.  Then $ \End(\mathbb{V}_{\bf d} P) $ becomes a graded ring so there is a grading on $ A_{\bf d}. $

In [Maz] it was shown that $ A_{\bf d}^{\mathfrak{p}} $ is a graded quotient of $ A_{\bf d}. $
Now we can consider $ \mathcal{O}_{\bf d}(\mathfrak{gl}_n) $ and  $ \mathcal{O}_{\bf d}^{\mathfrak{p}} (\mathfrak{gl}_n) $ as graded categories by considering
$ \text{gmod}-A_{\bf d} $ and $ \text{gmod}-A_{\bf d}^{\mathfrak{p}}  $ respectively.

We fix  a graded lift of a generalized Verma module $ \widetilde{M^{\mathfrak{p}}}(x.\omega_{\bf d}) $ so that its head is concentrated in degree zero.
The proof of proposition 5.2 of [FKS] gives a graded inclusion $ \widetilde{M}(x.\omega_{\bf d})\langle l(x) \rangle \rightarrow \widetilde{M}(\omega_{\bf d}) $ of Verma modules.  This implies that the standard maps in a generalized BGG resolution are homogeneous of degree 1.

Define the k-tuple $ {\bf d} + t_0\epsilon_0 + \cdots + t_{k-1}\epsilon_{k-1} $ to be $ (d_{k-1}+t_{k-1}, \ldots, d_{0}+t_{0}) $ where the $ t_i $ are integers.
Assume $ t>0. $ 
Define $ ({\bf d}; {\bf d} + t \epsilon_i- t \epsilon_{i-1}) $ to be the (k+1)-tuple $ (d_{k-1}, \ldots, d_i, t, d_{i-1}-t, d_{i-2}, \ldots, d_0). $
Define $ ({\bf d}; {\bf d} - t \epsilon_i+ t \epsilon_{i-1}) $ to be the (k+1)-tuple $ (d_{k-1}, \ldots, d_i-t, t, d_{i-1}, d_{i-2}, \ldots, d_0). $

Recall the definitions of $ \mathcal{E}_i, \mathcal{F}_i $ from section ~\ref{3.1}.
\begin{lemma}
\label{lemma30}
\begin{enumerate}
\item $ \mathcal{E}_{i} \colon \mathcal{O}_{\bf d} \rightarrow \mathcal{O}_{{\bf d} +\epsilon_i - \epsilon_{i-1}} $ is isomorphic to 
$ \theta_{{\bf d}; {\bf d} + \epsilon_i- \epsilon_{i-1}}^{{\bf d} + \epsilon_i- \epsilon_{i-1}} \theta^{{\bf d}; {\bf d} + \epsilon_i- \epsilon_{i-1}}_{\bf d}. $
\item $ \mathcal{F}_{i} \colon \mathcal{O}_{\bf d} \rightarrow \mathcal{O}_{{\bf d} -\epsilon_i + \epsilon_{i-1}} $ is isomorphic to 
$ \theta_{{\bf d}; {\bf d} - \epsilon_i+ \epsilon_{i-1}}^{{\bf d} - \epsilon_i+ \epsilon_{i-1}} \theta^{{\bf d}; {\bf d} - \epsilon_i+ \epsilon_{i-1}}_{\bf d}. $
\end{enumerate}
\end{lemma}

\begin{proof}
See proposition 3.2b of [FKS].
\end{proof}
 
Now we are prepared to introduce the graded lifts $ \widetilde{\mathcal{E}}_i, $ $ \widetilde{\mathcal{F}}_i, $ $ \widetilde{\mathcal{H}}, $ and $ \widetilde{\mathcal{H}}^{-1}. $
By lemma ~\ref{lemma28}, 
$$ \mathcal{E}_i \colon \text{mod}-A_{\bf d} \rightarrow \text{mod}-A_{{\bf d}+\epsilon_i-\epsilon_{i-1}} $$ is given by
$$ \bullet \otimes_{\text{mod}-A_{\bf d}} \Hom_{\mathfrak{gl}_n}(P_{{\bf d} + \epsilon_i- \epsilon_{i-1}}, \mathcal{E}_i P_{\bf d}). $$
By lemma ~\ref{lemma30} this is
$$ \bullet \otimes_{\text{mod}-A_{\bf d}} \Hom_{\mathfrak{gl}_n}(P_{{\bf d} + \epsilon_i- \epsilon_{i-1}}, 
\theta_{{\bf d}; {\bf d} + \epsilon_i- \epsilon_{i-1}}^{{\bf d} + \epsilon_i- \epsilon_{i-1}} \theta^{{\bf d}; {\bf d} + \epsilon_i- \epsilon_{i-1}}_{\bf d}P_{\bf d}). $$
Then by lemma ~\ref{lemma29} this is isomorphic to
$$ \bullet \otimes_{\text{gmod}-A_{\bf d}} \Hom_{C^{{\bf d} + \epsilon_i- \epsilon_{i-1}}}({\mathbb{V}}_{{\bf d} + \epsilon_i- \epsilon_{i-1}}P_{{\bf d} + \epsilon_i- \epsilon_{i-1}},
\text{Res}^{{\bf d} + \epsilon_i- \epsilon_{i-1}}_{{\bf d}; {{\bf d} + \epsilon_i- \epsilon_{i-1}}} C^{{\bf d}; {{\bf d} + \epsilon_i- \epsilon_{i-1}}} \otimes_{C^{\bf d}} 
{\mathbb{V}}_{\bf d}P_{\bf d}). $$

\begin{define}
\begin{enumerate}
\item Let 
$ \widetilde{\mathcal{E}}_i = $
$$ \bullet \otimes_{\text{gmod}-A_{\bf d}} \Hom_{C^{{\bf d} + \epsilon_i- \epsilon_{i-1}}}({\mathbb{V}}_{{\bf d} + \epsilon_i- \epsilon_{i-1}}P_{{\bf d} + \epsilon_i- \epsilon_{i-1}},
\text{Res}^{{\bf d} + \epsilon_i- \epsilon_{i-1}}_{{\bf d}; {{\bf d} + \epsilon_i- \epsilon_{i-1}}} C^{{\bf d}; {{\bf d} + \epsilon_i- \epsilon_{i-1}}} \otimes_{C^{\bf d}} 
{\mathbb{V}}_{\bf d}P_{\bf d}\langle 1-d_{i-1} \rangle). $$

\item Let 
$ \widetilde{\mathcal{F}}_i = $
$$ \bullet \otimes_{\text{gmod}-A_{\bf d}} \Hom_{C^{{\bf d} - \epsilon_i+ \epsilon_{i-1}}}({\mathbb{V}}_{{\bf d} - \epsilon_i+ \epsilon_{i-1}}P_{{\bf d} - \epsilon_i+ \epsilon_{i-1}},
\text{Res}^{{\bf d} - \epsilon_i+ \epsilon_{i-1}}_{{\bf d}; {{\bf d} - \epsilon_i+ \epsilon_{i-1}}} C^{{\bf d}; {{\bf d} - \epsilon_i+ \epsilon_{i-1}}} \otimes_{C^{\bf d}} 
{\mathbb{V}}_{\bf d}P_{\bf d}\langle 1-d_i \rangle). $$

\item $ \widetilde{\mathcal{H}} = \Id\langle d_i-d_{i-1} \rangle. $

\item $ \widetilde{\mathcal{H}}^{-1} = \Id\langle -(d_i-d_{i-1}) \rangle. $
\end{enumerate}
\end{define}

Note that it is possible to define graded lifts of functors categorifying divided powers $ \widetilde{\mathcal{E}}_i^{(t)}, \widetilde{\mathcal{F}}_i^{(t)}. $

\begin{theorem}
There are isomorphisms of graded projective functors:
\begin{enumerate}
\item $ \widetilde{\mathcal{H}}_i \widetilde{\mathcal{H}}_i^{-1} \cong \Id \cong \widetilde{\mathcal{H}}^{-1}_i \widetilde{\mathcal{H}}_i. $
\item $ \widetilde{\mathcal{H}}_i \widetilde{\mathcal{H}}_j \cong \widetilde{\mathcal{H}}_j \widetilde{\mathcal{H}}_i. $
\item If $ |i-j|>1, $ $ \widetilde{\mathcal{H}}_i \widetilde{\mathcal{E}}_j \cong \widetilde{\mathcal{E}}_j \widetilde{\mathcal{H}}_i. $
\item  $ \widetilde{\mathcal{H}}_i \widetilde{\mathcal{E}}_i \cong \widetilde{\mathcal{E}}_i \widetilde{\mathcal{H}}_i \langle 2 \rangle. $
\item If $ |i-j|=1, $ $ \widetilde{\mathcal{H}}_i \widetilde{\mathcal{E}}_j \cong \widetilde{\mathcal{E}}_j \widetilde{\mathcal{H}}_i \langle -1 \rangle. $
\item If $ |i-j|>1, $ $ \widetilde{\mathcal{H}}_i \widetilde{\mathcal{F}}_j \cong \widetilde{\mathcal{F}}_j \widetilde{\mathcal{H}}_i. $
\item  $ \widetilde{\mathcal{H}}_i \widetilde{\mathcal{F}}_i \cong \widetilde{\mathcal{F}}_i \widetilde{\mathcal{H}}_i \langle -2 \rangle. $
\item If $ |i-j|=1, $ $ \widetilde{\mathcal{H}}_i \widetilde{\mathcal{F}}_j \cong \widetilde{\mathcal{F}}_j \widetilde{\mathcal{H}}_i \langle 1 \rangle. $
\item If $ i \neq j, $ $ \widetilde{\mathcal{E}}_i \widetilde{\mathcal{F}}_j \cong \widetilde{\mathcal{F}}_j \widetilde{\mathcal{E}}_i. $
\item If $ |i-j|>1, $ $ \widetilde{\mathcal{E}}_i \widetilde{\mathcal{E}}_j \cong \widetilde{\mathcal{E}}_j \widetilde{\mathcal{E}}_i. $
\item If $ |i-j|>1, $ $ \widetilde{\mathcal{F}}_i \widetilde{\mathcal{F}}_j \cong \widetilde{\mathcal{F}}_j \widetilde{\mathcal{F}}_i. $
\item $ \widetilde{\mathcal{E}}_i \widetilde{\mathcal{E}}_i \widetilde{\mathcal{E}}_{i+1} \oplus \widetilde{\mathcal{E}}_{i+1} \widetilde{\mathcal{E}}_i
\widetilde{\mathcal{E}}_i \cong \widetilde{\mathcal{E}}_i \widetilde{\mathcal{E}}_{i+1} \widetilde{\mathcal{E}}_i \langle 1 \rangle \oplus
\widetilde{\mathcal{E}}_i \widetilde{\mathcal{E}}_{i+1} \widetilde{\mathcal{E}}_i \langle -1 \rangle. $
\item $ \widetilde{\mathcal{E}}_i \widetilde{\mathcal{E}}_i \widetilde{\mathcal{E}}_{i-1} \oplus \widetilde{\mathcal{E}}_{i-1} \widetilde{\mathcal{E}}_i
\widetilde{\mathcal{E}}_i \cong \widetilde{\mathcal{E}}_i \widetilde{\mathcal{E}}_{i-1} \widetilde{\mathcal{E}}_i \langle 1 \rangle \oplus
\widetilde{\mathcal{E}}_i \widetilde{\mathcal{E}}_{i-1} \widetilde{\mathcal{E}}_i \langle -1 \rangle. $
\item $ \widetilde{\mathcal{F}}_i \widetilde{\mathcal{F}}_i \widetilde{\mathcal{F}}_{i+1} \oplus \widetilde{\mathcal{F}}_{i+1} \widetilde{\mathcal{F}}_i
\widetilde{\mathcal{F}}_i \cong \widetilde{\mathcal{F}}_i \widetilde{\mathcal{F}}_{i+1} \widetilde{\mathcal{F}}_i \langle 1 \rangle \oplus
\widetilde{\mathcal{F}}_i \widetilde{\mathcal{F}}_{i+1} \widetilde{\mathcal{F}}_i \langle -1 \rangle. $
\item $ \widetilde{\mathcal{F}}_i \widetilde{\mathcal{F}}_i \widetilde{\mathcal{F}}_{i-1} \oplus \widetilde{\mathcal{F}}_{i-1} \widetilde{\mathcal{F}}_i
\widetilde{\mathcal{F}}_i \cong \widetilde{\mathcal{F}}_i \widetilde{\mathcal{F}}_{i-1} \widetilde{\mathcal{F}}_i \langle 1 \rangle \oplus
\widetilde{\mathcal{F}}_i \widetilde{\mathcal{F}}_{i-1} \widetilde{\mathcal{F}}_i \langle -1 \rangle. $
\item If $ d_{i-1} > d_i $ then
$$ \widetilde{\mathcal{F}}_i \widetilde{\mathcal{E}}_i \cong \widetilde{\mathcal{E}}_i \widetilde{\mathcal{F}}_i 
\oplus_{r=0}^{d_{i-1}-d_i-1} \Id \langle d_{i-1}-d_i-1-2r \rangle. $$
If $ d_{i} > d_{i-1} $ then
$$ \widetilde{\mathcal{E}}_i \widetilde{\mathcal{F}}_i \cong \widetilde{\mathcal{F}}_i \widetilde{\mathcal{E}}_i 
\oplus_{r=0}^{d_{i}-d_{i-1}-1} \Id \langle d_{i}-d_{i-1}-1-2r \rangle. $$		
If $ d_i = d_{i-1}, $ then $ \widetilde{\mathcal{F}}_i \widetilde{\mathcal{E}}_i \cong \widetilde{\mathcal{E}}_i \widetilde{\mathcal{F}}_i. $
\end{enumerate}
\end{theorem}

\begin{proof}
See [Su].
\end{proof}

The Zuckerman functor $ Z^{\mathfrak{p}} $ also becomes a graded functor $ \widetilde{Z}^{\mathfrak{p}}. $  It is the functor of tensoring with the graded bimodule $ A_{\bf d}^{\mathfrak{p}}\langle \dim(\mathfrak{b})-\dim(\mathfrak{p}) \rangle. $

\begin{prop}
\label{zuckgroth}
Let $ a_{i_1, \ldots, i_r} = \dim(\mathfrak{b})-\dim(\mathfrak{p}). $  Then on the Grothendieck group, 
\begin{enumerate}
\item $ [L \widetilde{Z}^{(i_1, \ldots, i_r)}[a_{i_1, \ldots, i_r}]]=\pi_{i_1} \otimes \cdots \otimes \pi_{i_r} $  
\item $ [\widetilde{\epsilon}_{(i_1, \ldots, i_r)}]=\iota_{i_1} \otimes \cdots \otimes \iota_{i_r}. $
\end{enumerate}
\end{prop}

\begin{proof}
\begin{enumerate}
\item This follows from computing the derived Zuckerman functor on Verma modules.
By proposition 5.5 of [ES], 
$$ LZ^{(i_1, \ldots, i_r)} M(a_1, \ldots, a_n) \cong M^{(i_1, \ldots, i_r)}(a_{j_1}, \ldots, a_{j_n})[l(w)], $$
or zero where 
$$ a_{j_1}>\cdots > a_{j_{i_1}}, \ldots, a_{j_{n-i_r+1}} > \cdots > a_{j_n}, $$
$$ \lbrace a_{i_1+ \cdots + i_{k-1}+1}, \ldots, a_{i_1+ \cdots + i_k} \rbrace = 
\lbrace a_{j_{i_1+ \cdots + i_{k-1}+1}}, \ldots, a_{j_{i_1+ \cdots + i_{k}}} \rbrace $$
for $ k=1, \ldots, r $
and $ w $ is the length of the permutation mapping $ (a_1, \ldots, a_n) $ to $ (a_{j_1}, \ldots, a_{j_n}). $
The graded case now follows by noting that Verma modules are Koszul.
That is, there is a projective resolution of the graded Verma module by graded projective modules such that the projective module in homological degree $ i $ is generated by its degree $ i $ internal grading.  
\item The BGG resolution gives a resolution of a generalized Verma modules by Verma modules.  It has the appropriate grading by noting that $ \widetilde{M}(x.\omega_{\bf d})\langle l(x) \rangle \rightarrow \widetilde{M}(\omega_{\bf d}) $ is a homogeneous map of degree zero.
\end{enumerate}
\end{proof}

Finally, we let $ \widetilde{T}^{\mathfrak{p}}(w.\omega_{\bf d}) $ be the graded lift of the tilting module $ T^{\mathfrak{p}}(w.\omega_{\bf d}) $ such that $ \widetilde{M}^{\mathfrak{p}}(w.\omega_{\bf d}) $ occurs as a submodule.  Graded lifts of tilting modules were introduced in [MO].

\section{Special bases in the Grothendieck group}

\subsection{The Hecke algebra}
Let $ \mathcal{A} $ be the associative algebra over $ \mathbb{Q}(q) $ with generators $ H_{s_i} $ for each simple reflection $ s_i $ in the Weyl group $ \mathbb{S}_n $ and relations
$ H_{s_i} H_{s_j} = H_{s_j} H_{s_i} $ for $ |i-j|>1, $ $H_{s_i} H_{s_{i+1}} H_{s_i} = H_{s_{i+1}} H_{s_i} H_{s_{i+1}}, $ and $ (H_{s_i}+q)(H_{s_i}-q^{-1}) = 0. $

Let $ W_Q $ be a Young subgroup of $ W. $  Let $ W^Q $ be the set of minimal length coset representatives of $ W/W_Q. $  There is an $ \mathcal{A}- $ module $ M^Q $ defined as follows.  As a $ \mathbb{Q}(v) $ vector space, $ M^Q $ has basis $ m_y $ where $ y \in W^Q. $
Then $ H_{s_i} m_y = $
\begin{eqnarray*} 
m_{s_i y}+(q^{-1}-q)m_y & \textrm{ if } &l(s_i y)<l(y) \\
m_{s_i y} & \textrm{ if } & l(s_i y)>l(s_i), s_i y \in W^Q \\
q^{-1} m_y & \textrm{if} &l(s_i y)>l(y), s_i y \notin W^Q.
\end{eqnarray*}

Let $ d \colon \mathcal{A} \rightarrow \mathcal{A} $ be the algebra homomorphism given by $ d(q) = q^{-1} $ and $ d(H_x) = H_{x^{-1}}^{-1}. $
We extend this involution to an involution on $ M^Q $ given by $ d(m_e) = m_e $ and $ d(hm_e) = d(h) m_e $ where $ h \in \mathcal{A}. $

\begin{theorem}
There is a basis $ \lbrace \underline{m_x} | x \in W^Q \rbrace $ for $ M^Q $  such that $ d(\underline{m_y}) = \underline{m_y} $ and 
$ \underline{m_y} = \Sigma_x n_{x,y} m_x $ where $ n_{x,y} \in q\mathbb{Z}[q], $ and $ n_{y,y}. $
\end{theorem}

\begin{proof}
See [Soe2] or [FKK] for example.
\end{proof}

Fix an integral dominant weight $ \lambda $ whose stabilizer is the Young subgroup $ W_Q. $
Let $ x, y $ be shortest coset representatives of $ W/W_Q. $
Let $ n_Q^{x,y}(i) = (P(y.\lambda),M(x.\lambda) \langle i \rangle). $
Let $ P_Q^{x,y}(t) = \Sigma_i n_Q^{x,y}(i)t^i. $
The following proposition is theorem 3.11.4 of [BGS].  Note that the formula given there differs from the formula in proposition ~\ref{nilcoh} because in [BGS], longest coset representatives are considered.

\begin{prop}
\label{nilcoh}
Let $ P_{x,y}(q) $ be the usual Kazhdan-Lusztig polynomial of [KL].  Then
$$ P_Q^{x,y}(q) = \Sigma_{z \in W_Q} (-1)^{l(z)} P_{xz, y}(q^{-2}) q^{l(y)-l(x)}. $$ 
\end{prop}

Denote by $ A_Q $ the block of graded category $ \mathcal{O}(\mathfrak{gl}_n) $ corresponding to the integral dominant weight $ \lambda $ whose stabilizer is $ Q. $  Let $ (V_{k-1}^{\otimes n})_Q $ denote the corresponding weight space of $ V_{k-1}^{\otimes n}. $
 Consider the homomorphism $ \alpha \colon [\text{gmod}-A_Q] \rightarrow M^Q $ given by $ \alpha([\widetilde{M}(\sigma.\lambda)]) = m_{\sigma}. $
There is also a homomorphism $ \beta \colon[\text{gmod}-A_Q] \rightarrow V_{k-1}^{\otimes n}[Q] $ given by $ \beta([\widetilde{M}(\sigma.\lambda)]) = v_{i_1} \otimes \cdots \otimes v_{i_n} $
where 
$$ (i_1, \ldots, i_n) = \sigma.(k-1, k-1, \ldots, 0, 0). $$
There is a map $ \phi \colon M^Q \rightarrow V_{k-1}^{\otimes n}[Q] $ given by $ \phi(m_{\sigma}) = v_{i_1} \otimes \cdots \otimes v_{i_n} $ where
$ (i_1, \ldots, i_n) = w_0 \sigma.(k-1, k-1, \ldots, 0, 0). $
Finally, there is the graded derived twisting functor corresponding to the longest element in the Weyl group 
$$ L\widetilde{T}_{w_0}\langle l(w_0^Q) \rangle \colon D^b(\text{gmod}-A_Q) \rightarrow
D^b(\text{gmod}-A_Q).  $$
For a definition of the twisting functor see [AS].
The following lemma could be found in [FKS].
\begin{lemma}
There is an equality of maps: 
$ \phi \circ \alpha = \beta \circ [L\widetilde{T}_{w_0} \langle l(w_0^Q) \rangle] \colon \mathbb{Q}(q) \otimes_{\mathbb{Z}[q,q^{-1}]} [D^b(\text{gmod}-A_Q)] \rightarrow V_{k-1}^{\otimes n}[Q]. $
\end{lemma}

The following is theorem 2.5' of [FKK].

\begin{lemma}
Let $ (j_1, \ldots, j_n) = w_0 \sigma.(k-1,k-1, \ldots, 0, 0). $  Then
$ \phi(\underline{m_{\sigma}}) = v_{j_1} \diamond \cdots \diamond v_{j_n}. $ 
\end{lemma}

\begin{prop}
\label{selfdual}
The indecomposable projective module $ \widetilde{P}(y.\lambda) $ descends to the element $ \underline{m_y} $ of the positive self dual basis of $ M^Q $ under the map $ \alpha $ above.
\end{prop}

\begin{proof}
By theorem 3.11.4 of [BGS], $ P_Q^{x,y}(q) $ gives the graded multiplicity of $ \widetilde{M}(x.\lambda) $ in $ \widetilde{P}(y.\lambda). $
On the other hand, by the proposition ~\ref{nilcoh}, 
$$ P_Q^{x,y}(q) = \Sigma_{z \in W_Q} (-1)^{l(z)} P_{xz, y}(q^{-2}) q^{l(y)-l(x)}. $$
In the notation of [Soe2], this is equal to 
$$ \Sigma_{z \in W_Q} {(-q)}^{l(z)} h_{xz,y} = n_{x,y}. $$
Recall from [Soe2] that $ \underline{m_y} = \Sigma_x n_{x,y} m_x. $
\end{proof}

\begin{theorem}
\label{tiltnonpar}
The indecomposable tilting module $ \widetilde{T}(a_1,\ldots,a_n) $ in $ \oplus_{\bf d} \text{gmod-}A_{\bf d} $ descends to the canonical basis element $ v_{a_1} \diamond \cdots \diamond v_{a_n} $ in the Grothendieck group under the map sending Verma modules to the standard basis.
\end{theorem}

\begin{proof}
By proposition ~\ref{selfdual}, indecomposable projective objects descend to the positive self dual basis in the Hecke module $ M^Q $ in the Grothendieck group.  Since $ \phi \circ \alpha = \beta \circ [L\widetilde{T}_{w_0} \langle l(w_0^Q) \rangle], $ and $ \phi $ maps the positive self dual basis to the canonical basis, $ L\widetilde{T}_{w_0} \langle l(w_0^Q) \rangle $ maps projective objects to objects which descend to the canonical basis in the Grothendieck group.  Proposition 5.2c of [FKS] shows that $ L\widetilde{T}_{w_0} \langle l(w_0^Q) \rangle \widetilde{P}(x.\lambda) \cong \widetilde{T}(w_0 x.\lambda). $
Thus tilting modules descend to the canonical basis via the map $ \beta. $
\end{proof}

\begin{example}
For the case of $ \mathcal{U}_q(\mathfrak{sl}_2), $ we have as in example 5.4 of [FKS], the identification $ [\widetilde{T}(1,0)] = v_1 \diamond v_0 = v_1 \otimes v_0 + q^{-1} v_0 \otimes v_1 $ and $ [\widetilde{T}(0,1)]=[\widetilde{M}(0,1)] = v_0 \diamond v_1 = v_0 \otimes v_1. $
\end{example}

Recall the definition of the involution $ \psi $ from section ~\ref{definitions}. 

\begin{lemma}
\label{commutes}
The maps $ \psi $ and $ \pi_{i_1, \ldots, i_r} $ commute:
$$ \psi \pi_{i_1, \ldots, i_r} = \pi_{i_1, \ldots, i_r} \psi \colon V_{k-1}^{\otimes i_1} \otimes \cdots \otimes V_{k-1}^{\otimes i_r} \rightarrow \Lambda^{i_1} V_{k-1} \otimes \cdots \otimes \Lambda^{i_r} V_{k-1}. $$
\end{lemma}

\begin{proof}
First consider the case $ r=1. $  The tilting modules in $ \oplus_{\bf d} \text{gmod-}A_{\bf d}^{i_1} $ descend to the canonical basis of $ V_{k-1}^{\otimes i_1}. $
Note that if this is not a regular block, then $ \text{gmod-}A_{\bf d}^{i_1} = 0. $
If $ w \neq w_{0}^{\bf d}, $ then by corollary ~\ref{zucktilt}, $ L\widetilde{Z}^{(i_1)} \widetilde{T}(w.\omega_{\bf d}) = 0. $  If $ w=w_0^{\bf d}, $ then in the Grothendieck group, $ [\widetilde{T}(w.\omega_{\bf d})]=v_{a_1} \otimes \cdots \otimes v_{a_{i_1}} $ with $ a_1 \leq \cdots \leq a_{i_1}. $  If any of these are equalities, 
$ \pi_{i_1} (v_{a_1} \otimes \cdots \otimes v_{a_{i_1}}) = 0. $
If they are all strict inequalities, then $ \pi_{i_1}(v_{a_1} \otimes \cdots \otimes v_{a_{i_1}}) = v_{a_{i_1}} \wedge \cdots \wedge v_{a_1}. $  Thus $ \pi_{i_1} $ sends canonical basis elements to canonical basis elements.

Now we see that
$$ \pi_{i_1, \ldots, i_r} \psi = (\pi_{i_1, \ldots, i_{r-1}} \otimes \pi_{i_r}) \circ \Theta(\psi \otimes \psi) $$
where the maps $ \psi $ on the right hand side are the involutions for their respective representations.
This is equal to 
$ \Theta \circ  (\pi_{i_1, \ldots, i_{r-1}} \otimes \pi_{i_r}) \circ (\psi \otimes \psi). $
By induction this is equal to
$$ \Theta \circ (\psi \pi_{i_1, \ldots, i_{r-1}} \otimes \psi \pi_{i_r}) = \psi \circ \pi_{i_1, \ldots, i_r}. $$
\end{proof}

\begin{prop}
\label{projcan}
For each element $ v'_{\diamond} $ in the canonical basis of $ \Lambda^{i_1} V_{k-1} \otimes \cdots \otimes \Lambda^{i_r} V_{k-1}, $ there exists an element $ v_{\diamond} $ in the canonical basis of $ V_{k-1}^{\otimes (i_1+\cdots+i_r)} $ such that
$ \pi_{i_1, \ldots, i_r}(v_{\diamond})=v'_{\diamond}. $
\end{prop}

\begin{proof}
Condition ~\ref{kl1} of an element to be in the canonical basis is satisfied due to lemma ~\ref{commutes}.  

Suppose $ a_1 < \cdots < a_{i_1}, \ldots, a_{n-i_r+1} < \cdots < a_n. $
Let
\begin{eqnarray*}
v &= &v_{a_1} \otimes \cdots \otimes v_{a_n} \\
v_{\diamond} &= &v_{a_1} \diamond \cdots \diamond v_{a_n} \\
v' &= &(v_{a_{i_1}} \wedge \cdots \wedge v_{a_1}) \otimes \cdots \otimes (v_{a_n} \wedge \cdots \wedge v_{a_{n-i_r+1}}) \\
v_{\diamond}' &= &(v_{a_{i_1}} \wedge \cdots \wedge v_{a_1}) \diamond \cdots \diamond (v_{a_n} \wedge \cdots \wedge v_{a_{n-i_r+1}})
\end{eqnarray*}

Then $ v_{\diamond} = v + \Sigma_{w \neq v} q^{-1}\mathbb{Z}[q^{-1}] w, $ where the summation is over the standard basis of $ V_{k-1}^{\otimes(i_1+\cdots+i_r)}. $
By the formulas for projection from section ~\ref{definitions}, 
$$ \pi_{i_1,\ldots,i_r} v_{\diamond} = v' + \Sigma_{w'} q^{-1}\mathbb{Z}[q^{-1}]w', $$
where the summation is over the standard basis of $ \Lambda^{i_1}V_{k-1} \otimes \cdots \otimes \Lambda^{i_r}V_{k-1}. $
The proof will be complete once it is verified that $ w' $ never equals $ v'. $

We denote the parabolic subalgebra determined by $ (i_1,\ldots,i_r) $ by $ \mathfrak{p}. $
Let
\begin{eqnarray*}
\widetilde{T}(v_{\diamond})  &= &\widetilde{T}(a_1,\cdots,a_n)\\
\widetilde{M}(v) &= &\widetilde{M}(a_1,\cdots,a_n)\\
\widetilde{T}^{\mathfrak{p}}(v_{\diamond}') &= &\widetilde{T}^{\mathfrak{p}}(a_{i_1},\ldots,a_1,\ldots,a_n,\ldots,a_{n-i_r+1})\\
\widetilde{M}^{\mathfrak{p}}(v') &= &\widetilde{M}^{\mathfrak{p}}(a_{i_1},\ldots,a_1,\ldots,a_n,\ldots,a_{n-i_r+1})
\end{eqnarray*}
Then
$$ \pi_{i_1,\ldots,i_r}(v_{\diamond})=[L\widetilde{Z}^{\mathfrak{p}}[a_{i_1,\ldots,i_r}]\widetilde{T}(v_{\diamond})] = [\widetilde{M}^{\mathfrak{p}}(v')]+q^{-1}\mathbb{Z}[q,q^{-1}] \Sigma_{w'} [\widetilde{M}^{\mathfrak{p}}(w')]. $$
Now note that $ L\widetilde{Z}^{\mathfrak{p}}[a_{i_1,\ldots,i_r}] \widetilde{T}(v_{\diamond}) $ is a (possibly shifted) tilting module $ \widetilde{T}^{\mathfrak{p}}(v_{\diamond}'). $  Thus a (possibly shifted) generalized Verma module $ \widetilde{M}^{\mathfrak{p}}(v') $ occurs exactly once in the generalized Verma flag.  Thus
$$ \pi_{i_1,\ldots,i_r} (v_{\diamond}) = v' + \Sigma_{w'\neq v'} q^{-1}\mathbb{Z}[q^{-1}]w'. $$
Therefore condition ~\ref{kl2} of section ~\ref{definitions} is satisfied.
\end{proof}

Recall the definition of $ \widetilde{T}^{\mathfrak{p}}(v_{\diamond}') $ from the proof of proposition ~\ref{projcan}.

\begin{theorem}
\label{tiltpar}
Indecomposable tilting objects $ \widetilde{T}^{\mathfrak{p}}(v_{\diamond}') $ in the parabolic subcategory $ \oplus_{\bf d} \text{gmod-} A_{\bf d}^{\mathfrak{p}} $ descend to the canonical basis in the Grothendieck group under the map sending generalized Verma modules to the standard basis.
\end{theorem}

\begin{proof}
Indecomposable tilting modules $ \widetilde{T}(v_{\diamond}) $ in $ \oplus_{\bf d} \text{gmod}-A_{\bf d} $ descend in the Grothendieck group to canonical basis elements of $ V_{k-1}^{\otimes(i_1+\cdots+i_r)}.$  The derived Zuckerman functor maps these tilting modules to indecomposable tilting modules $ \widetilde{T}^{\mathfrak{p}}(v_{\diamond}'). $ The theorem follows in light of propositions ~\ref{zuckgroth} and ~\ref{projcan}.
\end{proof}

\begin{theorem}
Let $ \lbrace b_1, \ldots, b_m \rbrace $ be the canonical basis for the tensor product of finite dimensional irreducible $ \mathcal{U}_q(\mathfrak{sl}_k) $ modules $ V(\lambda_1) \otimes \cdots \otimes V(\lambda_k).$ Let $ E_i, F_i, i=1,\ldots,k-1 $ be Chevalley generators for this algebra.  Then $ E_i b_j = \Sigma_n c_{i,j,n} b_n, F_i b_j = \Sigma_n d_{i, j, n} b_n $ where $ c_{i,j,n}, d_{i,j,n} \in \mathbb{N}[q,q^{-1}] $ for all $ i,j,n. $
\end{theorem}

\begin{proof}
First assume the representations in the tensor product are fundamental representations.
By theorem ~\ref{tiltpar}, the indecomposable tilting modules with specified graded lifts descend to the canonical basis of 
$ \Lambda^{i_1} V_{k-1} \otimes \cdots \otimes \Lambda^{i_r} V_{k-1}. $  
Since projective functors map tilting modules to tilting modules, the coefficients of the canonical basis under the action of $ E_i $ and $ F_i $ are certain multiplicities.  The positivity for the tensor products of fundamental representations now follows. 

A tensor product of finite dimensional, irreducible module may be embedded in a tensor product of fundamental representations.  Since a canonical basis for the tensor product of fundamental representations restricts to the canonical basis for a tensor product of the corresponding irreducible representations [Lus], we now have the positivity result in this more general case.

\end{proof}

\end{document}